\documentclass[11pt,letterpaper]{amsart}
\usepackage{amsxtra}
\usepackage[all]{xy}
\usepackage{palatino}
\usepackage{commath}
\usepackage{mathrsfs}
\usepackage{comment}
\usepackage{mathtools}

\usepackage{tikz-cd}

\usepackage{amssymb,amsmath,amsthm,epsf,epsfig,dsfont,bbm}

\usepackage{upref, eucal}

\theoremstyle{plain}
\newtheorem{theorem}{Theorem}[section]
\newtheorem{proposition}[theorem]{Proposition}
\newtheorem{lemma}[theorem]{Lemma}
\newtheorem{corollary}[theorem]{Corollary}
\newtheorem{conjecture}[theorem]{Conjecture}
\usepackage{booktabs}
\usepackage{multirow}
\usepackage{makecell}
\usepackage{array}
\allowdisplaybreaks   

\theoremstyle{definition}
\newtheorem*{definition}{Definition}

\theoremstyle{remark}
\newtheorem{remark}[theorem]{Remark}

\usepackage{amsfonts}
\usepackage{caption}
\usepackage{subcaption}
\usepackage{enumitem}
\usepackage{listings}
\usepackage[utf8]{inputenc}
\usepackage{mdframed}
\usepackage{lipsum}

\usepackage[normalem]{ulem}

\usepackage[bookmarks=false]{hyperref}   

\numberwithin{equation}{section}
\begin{document}

\newcommand\GSP{{\mathfrak {GSp}}}
\newcommand\SP{{\mathrm {Sp}}}
\newcommand\A{\mathbb{A}}
\newcommand\G{\mathbb{G}}
\newcommand\N{\mathbb{N}}
\newcommand\T{\mathbb{T}}
\newcommand\sO{\mathcal{O}}
\newcommand\sE{{\mathcal{E}}}
\newcommand\tE{{\mathbb{E}}}
\newcommand\sF{{\mathcal{F}}}
\newcommand\sG{{\mathcal{G}}}
\newcommand\GS{{\mathrm{GSp}}}
\newcommand\HH{{\mathrm H}}
\newcommand\mM{{\mathrm M}}
\newcommand\fS{\mathfrak{S}}
\newcommand\fP{\mathfrak{P}}
\newcommand\fQ{\mathfrak{Q}}
\newcommand\Qbar{{\bar{\bf Q}}}
\newcommand\sQ{{\mathcal{Q}}}
\newcommand\sP{{\mathbb{P}}}
\newcommand{\Q}{{\bf Q}}
\newcommand{\tH}{\mathbb{H}}
\newcommand{\Z}{{\bf Z}}
\newcommand{\R}{{\bf R}}
\newcommand{\C}{{\bf C}}
\newcommand{\g}{{\mathfrak{g}}}
\newcommand{\F}{\mathbb{F}}
\newcommand\gP{\mathfrak{P}}

\newcommand\SL{{\mathrm {SL}}}
\newcommand\SO{{\mathrm {SO}}}
\newcommand\gl{{\mathfrak {gl}}}
\newcommand\Gsp{{\mathrm {Gsp}}}
\newcommand\So{{\mathrm {Sp}}}
\newcommand\Sl{{\mathfrak {sl}}}
\newcommand{\legendre}[2] {\left(\frac{#1}{#2}\right)}
\newcommand\iso{{\> \simeq \>}}
\newcommand\Frob{{\mathrm {Frob}}}
\newtheorem{thm}{Theorem}
\newtheorem{cor}[thm]{Corollary}
\newtheorem{conj}[thm]{Conjecture}
\newtheorem{prop}[thm]{Proposition}

\theoremstyle{definition}

\theoremstyle{remark}
\newtheorem{claim}[thm]{Claim}

\newtheorem{lem}[thm]{Lemma}

\theoremstyle{definition}
\newtheorem{dfn}{Definition}

\theoremstyle{remark}
\newtheorem*{fact}{Fact}
\makeatletter
\def\imod#1{\allowbreak\mkern10mu({\operator@font mod}\,\,#1)}
\makeatother

\makeatletter
\def\imod#1{\allowbreak\mkern10mu({\operator@font mod}\,\,#1)}
\makeatother
\title[Algorithm to compute bounds of dimensions for Siegel Modular forms with Nebentypus]{An algorithm to compute upper bounds of dimensions for Siegel Modular Forms of Prime Level and Arbitrary Nebentypus}

\author[Dron Airon]{Dron Airon}
\address{Dron Airon \\ Department of Mathematics \\ INDIAN INSTITUTE OF SCIENCE EDUCATION AND RESEARCH, PUNE, INDIA}
\email{dron.airon@students.iiserpune.ac.in}
\author[Debargha Banerjee]{Debargha Banerjee}
\address{ Debargha Banerjee\\ Department of Mathematics \\ INDIAN INSTITUTE OF SCIENCE EDUCATION AND RESEARCH, PUNE, INDIA}
\email{debargha.banerjee@gmail.com}
\author[Ronit Debnath]{Ronit Debnath}
\address{ Ronit Debnath\\ Department of Mathematics \\ INDIAN INSTITUTE OF SCIENCE EDUCATION AND RESEARCH, PUNE, INDIA}
\email{ronitdebnath14@gmail.com}
\author[Pranjal Vishwakarma]{Pranjal Vishwakarma}
\address{Pranjal P\\Department of Mathematics \\ North Eastern Regional Institute of Science and Technology, India}
\email{pranjalvishwa.03@gmail.com}
\thanks{We thank Professor Cris Poor and David Yuen for fruitful email correspondence in the initial stage of the project.}

\issueinfo{}
  {}
  {}
  {}
\pagespan{1}{}
\copyrightinfo{}
  {Korean Mathematical Society}


 \begin{abstract}
We describe an algorithmic method to determine the image of restriction maps for Siegel modular forms with \textit{arbitrary} characters and arbitrary weights. 
A program has been implemented in the mathematical software \texttt{Java} to compute the Fourier expansion of the image of these restriction maps for Siegel modular forms of genus two. 
This approach allows us to compute an upper bound for the space of Siegel modular forms with 
non-trivial {\it characters} (which has not been previously known) and arbitrary weights
(including low weight $k \leq  4$).

\end{abstract}

\maketitle

\section{Introduction}
\label{sec:intro}
Let $S_2^k\bigl(\Gamma_0^{(2)}(l), \chi\bigr)$ denote the space of Siegel modular forms of weight $k$ and genus $g=2$, with nebentypus character $\chi$ and prime level $l$ (cf.~\S\ref{notation}).

In this article, we are interested in the upper bound of the dimension of these vector spaces for arbitrary weights
(including low weight $k \leq  4$).   
For classical modular forms of arbitrary weight and nebentypus, the dimensions~\cite[\S 5]{MR2112196} and Fourier coefficients can be explicitly calculated~\cite{lmfdb}. It is natural to wonder if we can write down the dimension formula and, more importantly, Fourier coefficients for Siegel modular forms of higher genus.

Currently, there is no general formula that expresses the dimensions of these vector spaces in terms of their levels, weights, and characters. More significantly, there is no natural method for determining the explicit Fourier coefficients of Siegel modular forms with arbitrary nebentypus. Nevertheless, recent progress in this direction has been made in the works of Böcherer, Schulze-Pillot, and Das.

A compilation of known dimensions can be found in~\cite{Schmidt}. Wakatsuki derived a dimension formula for cohomological Siegel modular weight forms $\geq 5$ and arbitrary vector-valued weights for prime level congruence subgroups with trivial nebentypus (see~\cite[Theorem~7.4]{MR2843308}). The computation of such low-weight cases ($k < 5$) was pioneered by a series of articles by Poor and Yuen~\cite{MR1795560, MR1898538, MR2379329} for a {\it prime level} and trivial nebentypus.
In this paper, we write an algorithm to determine the bounds of the dimensions of the spaces of Siegel modular forms for \textit{ prime level congruence subgroups with non-trivial character} and  arbitrary weights
(including low weight $k \leq  4$). 

In~\cite{MR2379329}, the authors describe the images of the restriction maps for Siegel modular forms of genus two for congruence subgroups of the form $\Gamma_0(p)$ for a prime level  $p$  and weights below~$5$. In this map, a Siegel modular form of level~$p$ is sent to a collection of classical elliptic modular forms of level~$pl$, where~$l$ runs over primes distinct from~$p$. The restriction map is explicitly written for the prime levels $p \leq 41$ and scalar weight~$4$, and is subsequently used~\cite[p.~66]{MR2379329} to compute the upper limits for the dimensions of the corresponding Siegel modular form spaces. Since there is no algorithmic method for general~$p$, the computation of Fourier coefficients associated with these restriction maps becomes very laborious. The corresponding lower bounds for these spaces can be obtained using Jacobi cusp forms and classical elliptic cusp forms of half-integral weight.  

We investigate an upper bound of the dimensions of these vector spaces. The broader objective of this work is to establish a \textit{Sturm-type bound} within the framework of Siegel modular forms. Specifically, the aim is to ascertain the minimal number of Fourier coefficients required to determine whether a given Siegel modular form vanishes identically. As the Fourier coefficients of Siegel modular forms are naturally indexed by symmetric matrices, it is appropriate to consider their \textit{dyadic traces} and to formulate corresponding bounds in terms of these traces.

The \textit{objective} of this paper is to develop an \textit{algorithm} that computes the image of the aforementioned restriction map for \textit{ arbitrary prime levels} and \textit{general sending matrices}. In particular, we provide an explicit description of the images of the restriction maps for Siegel modular forms with \textit{arbitrary nebentypus}.

The main result of the paper is the following theorem.

\begin{theorem}
 \label{MainTheorem}
 There exists an algorithm to compute the Fourier coefficients of the elliptic cusp forms that appear in the images of the restriction maps for prime level, arbitrary weight and arbitrary nebentypus.   The above algorithm can be applied to determine an explicit upper bound for the dimensions of the space of Siegel modular forms
\[
S_2^k\bigl(\Gamma_0^{(2)}(l), \chi\bigr),
\]
for arbitrary weights $k$ (including low weight $k \leq  4$), $l$ is a prime and $\chi$ is a non-trivial Dirichlet character.
\end{theorem}
We restrict our attention to the genus $g=2$ since the computation of dyadic traces (cf.~\ref{resttriv})
becomes complicated for the genus $g \geq 3$.  We restrict ourselves to prime levels, as the relevant decomposition for general levels is currently not available (cf.~Remark~\ref{remarkres}).

In Section 6, we provide an example that illustrates how to apply the algorithm to obtain an upper bound for a given level, nebentypus, and weight. Our calculations are carried out for the weight $k=2,3,4$ and for smaller primes.

\begin{conjecture}
   For any prime \(l\), we have

$$
\phi_s^*f\equiv 0
$$
for any $f\in S_2^k\bigl(\Gamma_0^{(2)}(l), \chi\bigr)$ such that $\chi^2\neq \pm1$ and
for every \(s\neq I_2\) such that \(l'=\det(s)\) is a prime that is not a quadratic residue modulo \(l\).

\end{conjecture}

\begin{theorem}
    For weight $k=2$ and nontrivial nebentypus as chosen, we have following dimensions;

\begin{table}[ht]
\centering
\small
\renewcommand{\arraystretch}{1.35}
\setlength{\tabcolsep}{8pt}
\setlength{\arrayrulewidth}{0.4pt}
\begin{tabular}{!{\vrule width 1pt}c|c|c!{\vrule width 1pt}}
\noalign{\hrule height 1pt}
$p$ & $\chi$ & $\dim S_2^2(\Gamma_0(p),\chi)$ \\
\noalign{\hrule height 1pt}
$5$  & $\chi(2)=\zeta_4$          & $0$ \\
\hline
$7$  & $\chi(3)=\zeta_6-1$        & $0$ \\
\hline
$13$ & $\chi(2)=\zeta_6-1$        & $0$ \\
\hline
$17$ & $\chi(3)=-\zeta_4$         & $0$ \\
\hline
$19$ & $\chi(2)=\zeta_6-1$        & $0$ \\
\noalign{\hrule height 1pt}
\end{tabular}
\caption{ $\dim S_2^2(\Gamma_0(p),\chi)$ for $p\leq 19$, non-trivial characters only.}
\label{tab:dim_bounds_short}
\end{table}

\end{theorem}
\newpage
\begin{theorem}
    For weight $k=3$ and chosen nebentypus, we have following dimensions;

\begin{table}[ht]
\centering
\small
\renewcommand{\arraystretch}{1.6}
\setlength{\tabcolsep}{8pt}
\setlength{\arrayrulewidth}{0.4pt}
\begin{tabular}{!{\vrule width 1pt}c|c|c!{\vrule width 1pt}}
\noalign{\hrule height 1pt}
$p$ & $\chi$ & $\dim S_2^3(\Gamma_0(p),\chi)$ \\
\noalign{\hrule height 1pt}
$5$  & for all char. & $0$ \\
\hline
$7$  & for all char. & $0$ \\
\hline
\multirow{3}{*}{$11$}
  & $\chi(2)=-\zeta_{10}^{3}$ & $0$ \\
\cline{2-3}
  & $\chi(2)=\zeta_{10}$      & $0$ \\
\cline{2-3}
  & $\chi(2)=-1$              & $0$ \\
\noalign{\hrule height 1pt}
\end{tabular}
\caption{ $\dim S_2^3(\Gamma_0(p),\chi)$ for $p\in\{2,3,5,7,11\}$.}
\label{tab:dim_bounds_wt3_short}
\end{table}

\end{theorem}

\begin{theorem}
    For weight $k=4$ and chosen nebentypus, we have following dimensions;

\begin{table}[ht]
\centering
\small
\renewcommand{\arraystretch}{1.6}
\setlength{\tabcolsep}{8pt}
\setlength{\arrayrulewidth}{0.4pt}
\begin{tabular}{!{\vrule width 1pt}c|c|c!{\vrule width 1pt}}
\noalign{\hrule height 1pt}
$p$ & $\chi$ & $\dim S_2^4(\Gamma_0(p),\chi)$ \\
\noalign{\hrule height 1pt}
$3$ & $\chi(2)=-1$ & $0$ \\
\hline
\multirow{2}{*}{$5$}
  & $\chi(2)=-1$    & $1$ \\
\cline{2-3}
  & $\chi(2)=\iota$ & $0$ \\
\hline
$7$ & $\chi(3)=\zeta_6-1$ & $2$ \\
\noalign{\hrule height 1pt}
\end{tabular}
\caption{ $\dim S_2^4(\Gamma_0(p),\chi)$ for $p\in\{3,5,7\}$.}
\label{tab:dim_bounds_357_short}
\end{table}

\end{theorem}

\section{Notation}
\label{notation}
Let $N, k \in \mathbb{Z}^+$. 
\begin{itemize}
    \item For genus $g\geq1$ we define the upper half Seigel space to be 
    \[
    \mathbb{H}_{2g} := \{ \Omega \in \mathbb{C}^{g \times g} \mid \Omega \text{ is symmetric and } \operatorname{Im}(\Omega) > 0 \}
    \]
    denotes the upper half-space of Siegel.

    \item 
    Let \(R\) be a commutative ring with 1. We denote by \(1_g\) and \(0_g\) the identity and zero matrix of the ring \(M_g(R)\) and now consider the matrix 
\[
J_g = \begin{pmatrix}
0 & 1_g \\
-1_g & 0
\end{pmatrix}.
\]

    \item 
    We denote by \(\mathrm{GSp}_{g}(R)\) the algebraic group of symplectic similitudes with respect to \(J_g\). Hence,
\[
\mathrm{GSp}_{g}(R) = \{M \in \mathrm{GL}_{2g}(R) \mid M^t J_g M = \mu(M) J_g \}
\]
   \item 
   For \(g = 1\) we have \(\mathrm{GSp}_2 = \mathrm{GL}_2\). The map:
\[
M \to \mu(M)
\]
defines a character \(\mu : \mathrm{GSp}_{g}(R) \to R^*\). We refer to \(\mu\) as the similarity factor.
    \item 
    $\Gamma_2 = \mathrm{Sp}_2(\mathbb{Z})$ denotes the Siegel modular group.

\item 
Let $\mathfrak{X}_2$ be the set of positive definite matrices $2 \times 2$ with integral values, half-integral.
    \item For any $N\in \N$, consider the congruence subgroup 
    
    \[
    \Gamma_0^{(2)}(N) = \left\{ 
    \begin{pmatrix}
        A & B \\ C & D
    \end{pmatrix}
    \in \mathrm{Sp}_2(\mathbb{Z}) \; \middle| \; C \equiv 0 \pmod{N} \right\}.
    \]
\item 
If \(g = 2\), then the above group \(\mathrm{GSp}_2\) is the symplectic similitude group defined by the set of matrices \(h\) in \(\mathrm{GL}_4\) that satisfy \(h^t J h = c(h) J\) for
\[
J = \begin{pmatrix}
0 & 0 & -1 & 0 \\
0 & 0 & 0 & -1 \\
1 & 0 & 0 & 0 \\
0 & 1 & 0 & 0
\end{pmatrix}
\]
and some \(c(h) \in \mathbb{G}_m\). We define a map \(c : \mathrm{GSp}_2 \to \mathbb{G}_m\) by letting \(g \in \mathrm{GSp}_2\) map to \(c(h)\). This is the similitude character, and its kernel is the symplectic group \(\mathrm{Sp}_2\).
    \item For $\Gamma \subseteq \Gamma_2$ of the finite index, the space of Siegel modular forms of weight $k$ with respect to $\Gamma$ is denoted by $M_2^k(\Gamma)$, and the subspace of cusp forms is denoted by $S_2^k(\Gamma)$.

    \item Let $\tilde{\chi}: (\mathbb{Z}/N\mathbb{Z})^{\times} \rightarrow \C^{\times}$ be a Dirichlet character. Let $\chi $ be a character in $\Gamma_0^{(2)}(N)$ as defined by Andrianov~\cite[section 1.3]{andrianov2009introduction}  
    \[
        \chi(M) := \tilde{\chi}(\det(D)), \quad 
        \text{where } 
        M = 
        \begin{pmatrix}
            A & B \\ C & D
        \end{pmatrix}
        \in \Gamma_0^{(2)}(N).
    \]
    Then, the space of Siegel modular forms with character $\chi$ is defined as
    \[
        S_2^k(\Gamma_0^{(2)}(N), \chi) = 
        \left\{ F \Bigm| F|_M = \chi(M) F \right\}.
    \] 
    where $F|_M$ is weight $k$ operator for Seigel forms defined as 
    \[
    F|_M(Z)= {\det(CZ+D)}^{-k}F(MZ).
    \]

    \item Let $V_2(\mathbb{Z})$ be the space of symmetric $2 \times 2$ matrices over $\mathbb{Z}$.  
    For $S \in V_2(\mathbb{Z})$, define
    \[
        t(S) = 
        \begin{pmatrix}
            I_2 & S \\ 0_2 & I_2
        \end{pmatrix}
        \in \mathrm{Sp}_2(\mathbb{Z}).
    \]

    \item When $t(V_2(\mathbb{Z})) \subseteq \Gamma$, we have the Fourier expansion
    \[
        f(\Omega) = 
        \sum_{t \in \mathfrak{X}_2}
        a(t; f)\, e(\langle t, \Omega \rangle),
    \]
    for $f \in S_2^k(\Gamma)$.  
    Here, $\langle t, \Omega \rangle = \operatorname{tr}(t\Omega)$, $e(z) = e^{2\pi i z}$.

    \item For $T, u \in \mathrm{GL}_2(\mathbb{R})$, define $T[u] := u^{t} T u$. 

    \item Consider the following matrices $E_0=\begin{pmatrix}
        I_2 & 0 \\0 & I_2
    \end{pmatrix}$ ; $E_2=\begin{pmatrix}
        0 & I_2 \\-I_2 & 0
    \end{pmatrix}$.
\end{itemize}

The remaining notations in the text will be introduced as required.

\section{The Restriction maps for trivial nebentypus}
\label{resttriv}
In this section following Poor- Yuen~\cite{MR2379329}, we recall how to compute the restriction maps for \textit{Siegel modular forms with trivial nebentypus}, which forms the basis for our subsequent extensions. We then apply their framework to compute new cases using the software \texttt{Java}.

For any natural number $N \in \mathbb{N}$, let $S_1^{k}(\Gamma_0(N))$ be the space of all classical elliptic cusp forms \cite{MR2112196} of weight $k$
for the congruence subgroup $\Gamma_0(N)$.  

Define the set of reduced matrices as follows:
\[
\mathfrak{X}_2^{\mathrm{red}} := 
\left\{
\begin{pmatrix}
    a & b \\
    b & c
\end{pmatrix}
\in \mathfrak{X}_2 \ \middle|\  0 \leq 2b \leq a \leq c
\right\}.
\]

We denote the symmetric matrix $\begin{pmatrix}
    a & b \\
    b & c
\end{pmatrix} \in \mathfrak{X}_2^{\mathrm{red}}$ by $[a^b c]$.  
For such a matrix $[a^b c]$, define the dyadic trace
\[
w([a^b c]) := \frac{1}{2}(a + c - |b|).
\]
Let $p$ be a prime. For $F \in S_2^k(\Gamma_0^{(2)}(p))$, consider the set of matrices $[a^b c]$ satisfying 
\[
w([a^b c]) < \tfrac{1}{6}(1 + p)k.
\]

We write $v \in [t]$ if there exists a matrix $m \in \mathrm{GL}_2(\mathbb{Z})$ such that $v = m t m^{T}$.  
Following \cite{MR2379329}, consider the quantity,
\[
v(j, s, t) = \#\{v \in [t] \mid \langle v, s \rangle = j\}.
\] where $ \langle A, B \rangle=trace(AB^t)$, which is the usual inner product on the space.\\

Following Poor-Yuen \cite{MR1898538}, we  define the map $\phi_s : \mathbb{H}_1 \to \mathbb{H}_2$ by $\phi_s(\tau) = s \tau$, where $s$ is a symmetric positive definite $2 \times 2$ matrix.  
Let $p \in \mathbb{N}$ be a prime.  
For $f \in S_2^{k}(\Gamma_0^{(2)}(p))$, the above map induces the restriction map $\phi_s^*(f) \in S_1^{2k}(\Gamma_0(p l))$, whose Fourier expansion is given by
\[
(\phi_s^* f)(\tau)
= \sum_{j \in \mathbb{N}}
\left(
\sum_{t \in \mathfrak{X}_2 : \langle t, s \rangle = j} a_0(t)
\right) q^j
= 
\sum_{j \in \mathbb{N}}
\left(
\sum_{t \in \mathfrak{X}_2^{\mathrm{red}}} v(j, s, t) a_0(t)
\right) q^j.
\]
For $c \equiv 0 \pmod{p l}$ and $l = \det(s)$, we have
\[
(\phi_s^*(f))|_{2k}
\begin{pmatrix}
    a & b \\
    c & d
\end{pmatrix}
(\tau)
= f|_k
\left(
\begin{pmatrix}
a I_2 & b s \\
c s^{-1} & d I_2
\end{pmatrix}
\right)
(s \tau)
= f|_k(s \tau)
= \phi_s^*(f)(\tau).
\]
This is well defined for $l = \det(s)$, and therefore, since $p l \mid c$, we also have $c s^{-1} \equiv 0 \pmod{p}$.

\subsection{\textbf{Upper Bound of the Dyadic Traces of the Set of Determining Coefficients}}
We study restriction maps ~\label{restriction}
\[
\phi_s^{\star}: S_2^2(\Gamma_0^{(2)}(p)) \longrightarrow S_1^4(\Gamma_0(p l))
\]
for $l \nmid p$.  
For $p = 43$ (a prime level for which the upper bound of the dimension is not known), we take $l \in \{1, 2, 3, 5\}$.  
For each choice of $l$, we select a $2 \times 2$ matrix $s$ such that $\det(s) = l$.  
For example,
\[
s = \begin{pmatrix}
    1 & 0 \\
    0 & 1
\end{pmatrix} \text{ for } l = 1, \quad
s = \begin{pmatrix}
    1 & 0 \\
    0 & 2
\end{pmatrix} \text{ for } l = 2, \quad \]

\[s = \begin{pmatrix}
    2 & 1 \\
    1 & 2
\end{pmatrix} \text{ for } l = 3, \quad 
s = \begin{pmatrix}
    2 & 1 \\
    1 & 3
\end{pmatrix} \text{ for } l = 5.
\]
Let $w([t])$ denote the dyadic trace of the matrix $[t]$.

\begin{definition}[Set of Determining Coefficients]
For a Siegel modular form $F$, let $a_i(F)$ denote the $i$ Fourier coefficients of $F$.
For two Siegel modular forms of a given level $N$ and weight $k$, the set of determining coefficients $S$ is defined as the set of coefficients such that if $F, G \in S_2^k(\Gamma_0^{(2)}(N), \chi)$ satisfies $a_i(F) = a_i(G)$ for all $i \in S$, then $F = G$.  

\end{definition}

Our starting point is \cite[Theorem~3.1]{MR2379329}. In that work, it is shown that for a Siegel modular form $F \in S_2^k(\Gamma_0^{(2)}(p))$ of level $N = p$, the set of determining coefficients consists of matrices $[t]$ with dyadic traces that satisfy
\[
w([t]) < \frac{1}{6}(1 + p)k.
\]
That is, equality of the coefficients in this set implies equality of the respective forms as a whole. 
Klein generalized this result to the prime powers \cite[Page~103]{Kleinthesis}.

\begin{lemma}
\label{determininglemma}
The set of determining coefficients for the space $S_2^k(\Gamma_0^{(2)}(p^i))$ consists of matrices $[t]$ with dyadic traces that satisfy:
\begin{itemize}
    \item $w([t]) < \dfrac{1}{6}(1 + p)k$ for $i = 1$, 
    \item $w([t]) < \dfrac{3}{2} + p^i\!\left(\dfrac{2k}{2\sqrt{3}\pi} - \dfrac{3}{2p^i}\right) \displaystyle\prod_{j = 1}^{i}\left(1 + \dfrac{1}{p^j}\right)$ for $i > 1$.
\end{itemize}
\end{lemma}

\begin{proof}
The result follows directly from \cite[Theorem~3.1]{MR2379329} and \cite[Page~103]{Kleinthesis}.
\end{proof}

\section{The Restriction maps for non-trivial nebentypus} 

Until now, we have discussed the space of Siegel modular forms with a trivial character. 
In this section, we develop \textit{restriction map} for the space of Siegel modular forms with {\it non-trivial} character. 
Let $f \in S^k_2(\Gamma_0^{(2)}(N), \chi)$ be a Siegel cusp form of genus two, weight $k$, level $N$ and nebentypus character $\chi$ and $l$ be a prime.  

\begin{lemma}
    For a cusp form $f \in S^k_2(\Gamma_0^{(2)}(N), \chi)$ and $d \in (\Z/l'N\Z)^{\times}$, we have 
    \[
    <d>\phi_s^*(f)   = \chi(dI_2)\phi_s^*(f),
    \]
    where $l'=\det(s)$. 
\end{lemma}
\begin{proof}
Define 
\[
\phi_s : \mathbb{H}_1 \to \mathbb{H}_2, 
\qquad 
\phi_s(\tau) = s\tau,
\]
where $s$ is a positive-definite matrix and $l' = \det(s)$.  
This induces a map on the corresponding Seigel cusp forms
\[
\phi_s^* : S_2^k(\Gamma_1^{(2)}(N)) \to S_1^{2k}(\Gamma_1(Nl')).
\]

The diamond operator 
\[
\langle d \rangle : S_1^{2k}(\Gamma_1(Nl')) \to S_1^{2k}(\Gamma_1(Nl'))
\]
acts by 
\[
\langle d \rangle f = f|[\alpha]_{2k},
\]
where 
\[
\alpha = \begin{pmatrix} a & b \\ c & \delta \end{pmatrix}, 
\qquad \delta \equiv d \pmod{Nl'}.
\]

Now for $f \in S_2^k(\Gamma_1^{(2)}(N))$, we have
\[
\langle d \rangle \phi_s^* f 
= \phi_s^* f | [\alpha]_{2k}, 
\qquad \alpha \in \Gamma_0(Nl').
\]

By the definition of the diamond operator,
\[
\phi_s^* f | [\alpha]_{2k}(\tau) 
= (c\tau + \delta)^{-2k} \, \phi_s^* f(\alpha \tau).
\]

Using \cite[Proposition 2.1]{MR1898538}, we get
\[
\phi_s^* f(\alpha \tau) 
= f\!\left(\tfrac{a\tau + b}{c\tau + \delta} s\right) 
= f\!\left( (a\tau s + bs)(c s^{-1}\tau s + \delta I_2)^{-1} \right).
\]

That is,
\[
\phi_s^* f(\alpha \tau) 
= f\!\left( 
\begin{pmatrix}
aI_2 & bs \\
c s^{-1} & \delta I_2
\end{pmatrix} \cdot \tau s
\right).
\]

Therefore,
\[
\phi_s^* f(\alpha \tau)
= \det(c s^{-1}\tau s + \delta I_2)^{k} 
\, f|_k
\begin{pmatrix}
aI_2 & bs \\
c s^{-1} & \delta I_2
\end{pmatrix}(\tau s).
\]

So,
\[
(c\tau + \delta)^{-2k} \, \phi_s^* f(\alpha \tau) 
= (c\tau + \delta)^{-2k} \det(c s^{-1}\tau s + \delta I_2)^{k}
\, f|_k
\begin{pmatrix}
aI_2 & bs \\
c s^{-1} & \delta I_2
\end{pmatrix}(\tau s).
\]

This simplifies to
\[
f|_k
\begin{pmatrix}
aI_2 & bs \\
c s^{-1} & \delta I_2
\end{pmatrix}(\tau s).
\]

Since $f \in S_2^k(\Gamma_0^{(2)}(N), \chi)$, we know
\[
f|_k M = \chi(M) f,
\qquad 
M = \begin{pmatrix} A & B \\ C & D \end{pmatrix}, 
\quad \chi(M) = \chi(\det D).
\]

Thus,
\[
f|_k
\begin{pmatrix}
aI_2 & bs \\
c s^{-1} & \delta I_2
\end{pmatrix}(\tau s)
= \chi(\det \delta) \, f(\tau s) 
= \chi(\det \delta) \, \phi_s^* f(\tau),
\]
since $c s^{-1} \equiv 0_2 \pmod{N}$ and
\[
\begin{pmatrix}
aI_2 & bs \\
c s^{-1} & \delta I_2 
\end{pmatrix}
\equiv
\begin{pmatrix}
aI_2 & * \\
0 & a^{-1} I_2
\end{pmatrix} \pmod{N}.
\]
\end{proof}

By the above lemma, when we restrict the map induced by the restriction map $\phi_s^*$ to the eigenspace $S_2^k(\Gamma_0^{(2)}(l), \chi_l)$, we obtain the following:
\[
    \phi_s^* : S_2^k(\Gamma_0^{(2)}(l), \chi_l) \longrightarrow S_1^{2k}(\Gamma_0(l'l), {\chi_l}^2\cdot {\chi_{l'}}^2),
\]
where ${\chi}_l$ is a character over $\mathbb{Z}/l\mathbb{Z}$, and 
\[
    \chi(M) := {\chi}_l(\det D),
\]
while ${\chi_l}^2\cdot\chi_{l'}^2$ denotes the extension of ${\chi}_l$ to $\mathbb{Z}/ll'\mathbb{Z}$ such that $\chi_{l'}$ is quadratic or trivial.
This last statement is an easy consequence due to above lemma.

Since $\chi_l(dI_2)=\chi_l(d^2)={\chi(d)}^2$. \\
From here on we will denote the character for classical forms coming from Siegel forms as $\chi'=\chi^2$ to avoid confusion.

\quad In the next proposition, we observe that the upper bound for the set of \textit{determining coefficients} obtained in Lemma 3.2, is independent of the nebentypus character $\chi$.

\begin{proposition}

The upper bound of the dyadic traces of the set of determining coefficients for the space $S_2^k(\Gamma_0^{(2)}(N), \chi)$ is the same as that for the space $S_2^k(\Gamma_0^{(2)}(N))$.
\end{proposition}

\begin{proof}
Let $F, G \in S_2^k(\Gamma_0^{(2)}(N), \chi)$, and consider the Siegel modular form $H = F - G$. The Fourier coefficients of $H$ vanish for all matrices $t$ with dyadic traces $w(t) < M$ for some natural number $M$. 

If $\chi$ is a Dirichlet character of order $r$, then $H^r \in S_2^{kr}(\Gamma_0^{(2)}(N))$, and $a_0(t) \equiv 0$ for all $t$ with $w(t) < Mr$. Hence $H^r \equiv 0$, since all relevant coefficients vanish for the corresponding bound and therefore $H \equiv 0$. Under these conditions, we conclude that $F = G$.

\end{proof}
Let $W_l$ be the Atkin-Lehner operator at the prime $l$. Note that the Atkin-Lehner operator has a decomposition 
\[
W_l=\begin{pmatrix}
        lx & y \\ ll'z & lw
    \end{pmatrix}=\begin{pmatrix}
        x & y \\ l'z & lw
    \end{pmatrix} \begin{pmatrix}
        l & 0 \\ 0 & 1
    \end{pmatrix}
\]    
with $\begin{pmatrix}
        x & y \\ l'z & lw
 \end{pmatrix}\in \SL_2(\mathbb{Z})$, with $l\nmid l'$. Let $E_2$ be as in the notation~\S~\ref{notation}.
\begin{proposition}
       Let $f\in{S_2}^k(\Gamma_0^{(2)}(l),\chi)$ and $s$ be a positive-definite matrix with $l'$ its determinant.  Then 
    $$W_l{\phi_s}^*f(\tau)= l^{k} \chi(l'z^2)\phi_{ls}^*(f|_{E_2})(\tau)=l^{k} \chi'(t)\phi_{ls}^*(f|_{E_2})(\tau)$$
    where t is such that $t^2\equiv l'z^2 \mod l$.
    \end{proposition}

\begin{proof}
Let 
\[
\phi_s^* : S_2^k\!\big(\Gamma_0^{(2)}(l),\chi\big) 
\;\longrightarrow\; S_1^{2k}\!\big(\Gamma_0(ll'),\chi'\big).
\]
By definition,
\[
W_l \phi_s^* f(\tau) 
= \phi_s^* f \big| W_l (\tau) 
= \phi_s^* f \,\Bigg|\,
\begin{bmatrix}
x & y \\ l'z & lw
\end{bmatrix}
\begin{bmatrix}
l & 0 \\ 0 & 1
\end{bmatrix}(\tau).
\]

Now, by the definition of the weight-$k$ operator for $\gamma \in GL_2(\mathbb{Q})$, we have
\[
\phi_s^* f \big| W_l (\tau) 
= \phi_s^* f \,\Bigg|\,
\begin{bmatrix}
x & y \\ l'z & lw
\end{bmatrix}
\begin{bmatrix}
l & 0 \\ 0 & 1
\end{bmatrix}(\tau)
= l^k\, \phi_s^* f \,\Bigg|\,
\begin{bmatrix}
x & y \\ l'z & lw
\end{bmatrix}(l\tau).
\]

 Using \cite[Proposition 4.1]{MR2379329}, we can write
\begin{equation}
\label{Atkindecomposition}
\begin{bmatrix}
x I_{2} & y s \\[2mm]
l' z s^{-1} & l w I_{2}
\end{bmatrix}
=
M E_{2}
\begin{bmatrix}
A & B \\[2mm]
0 & A^{*}
\end{bmatrix}.
\end{equation}

where $M \in \Gamma_0^{(2)}(l)$, 
\[
E_2 = \begin{bmatrix}
0 & I_2 \\ -I_2 & 0
\end{bmatrix}, 
\qquad 
\begin{bmatrix}
A & B \\ 0 & A^*
\end{bmatrix} \in Sp_4(\mathbb{Z}).
\]

A direct computation with $B = lI_2$ gives
\[
M =
\begin{bmatrix}
-lxI_2 + ysA & -xA^* \\[6pt]
lwA - ll' z s^{-1} & -l'zs^{-1}A^*
\end{bmatrix},
\]
and hence $M \in \Gamma_0^{(2)}(l)$.

Therefore,
\[
\begin{aligned}
l^k\,\phi_s^* f \,\Bigg|\,
\begin{bmatrix}
x & y \\ l'z & lw
\end{bmatrix}(l\tau)
&= l^k\, f \,\Bigg|\,
\begin{bmatrix}
xI_2 & ys \\ l'zs^{-1} & lwI_2
\end{bmatrix}(sl\tau) \\[6pt]
&= l^k\, f \,\Bigg|\, 
M E_2 
\begin{bmatrix}
A & lI_2 \\ 0 & A^*
\end{bmatrix}(sl\tau).
\end{aligned}
\]

Using Proposition 4.1, $f|_M = \chi(M) f$. Thus,
\[
l^k \chi(M)\,
f \,\Bigg|\,
E_2
\begin{bmatrix}
A & lI_2 \\ 0 & A^*
\end{bmatrix}(sl\tau)
= l^k \chi(M)\,
f \,\Bigg| E_2 (Asl\tau A^T + lA^T).
\]

Without loss of generality, set $A = I_2$. Now,
\begin{equation}
\begin{split}
W_l \phi_s^* f(\tau) 
&= l^k \chi(l'z^2)\,
\big(f|_{E_2}\big)(ls\tau + lI_2) \\[6pt]
&= l^k \chi(l'z^2)\,
\Big(f \,\Big|\, 
E_2
\begin{bmatrix}
I_2 & lI_2 \\ 0_2 & I_2
\end{bmatrix}\Big)(ls\tau) \\[6pt]
&= l^k \chi(l'z^2)\,
\Big(f \,\Big|\, 
\begin{bmatrix}
I_2 & 0 \\ -lI_2 & I_2
\end{bmatrix} E_2\Big)(ls\tau) \\[6pt]
&= l^k \chi(l'z^2)\,
\big(f|_{E_2}\big)(ls\tau) \\[6pt]
&= l^k \chi(l'z^2)\,
\phi_{ls}^*\big(f|_{E_2}\big)(\tau).
\end{split}
\end{equation}
\end{proof}

\begin{remark}
\label{remarkres}
We restrict ourselves to prime levels because the decomposition of \(W_{l}\) used in 
Equation~\ref{Atkindecomposition} is not known for general levels. This equation in turn depends on~\cite[Proposition 3.4]{MR1255043}. For general level, there will be more cusps and we need to understand Fourier expansion at more cusps. 

\end{remark}

\begin{proposition}

      Let $f\in S_2^k(\Gamma_0^{(2)}(l),\chi)$ be a Siegel modular form of level $l$. Let  $l'$ be a natural number such that $l'=det(s)$. Furthermore, assume $\chi$ to be $l'-$primary.  Then the elliptic modular form $\phi_s^*(f)$ satisfies the following equality:
   \[
   \phi_s^*(f) | W_{l'} = \chi(det((1-l\hat{l})s^{-1}))\phi_s^*f\\={\chi(l')}^{-1}\phi_s^*f={\chi'(t)}^{-1}\phi_s^*f;
   \] 

   where $\hat{l}$ is obtained using $l\hat{l} \equiv 1 \pmod{l'}$ and $t^2\equiv l' \mod l$
\end{proposition}

\begin{proof}
We first observe the decomposition
\[
\begin{bmatrix}
1 & 0 \\
c & 1
\end{bmatrix}
=
\begin{bmatrix}
\frac{ll'}{c} & -\hat{c} \\
ll' & 1 - \hat{c}c
\end{bmatrix}
\begin{bmatrix}
1 & \hat{c} \\
0 & \tfrac{ll'}{c}
\end{bmatrix}
\cdot \frac{c}{ll'},
\]
where the matrix
\[
W_{\tilde{c}} = 
\begin{bmatrix}
\frac{ll'}{c} & -\hat{c} \\
ll' & 1 - \hat{c}c
\end{bmatrix}
\]
is the Atkin-Lehner involution with $\tilde{c} = \tfrac{ll'}{c}$.

For $g \in M_{1}^{2k}\!\left(\Gamma_{0}(l'l)\right)$, we have
\[
g \,\Big|_{2k}
\begin{bmatrix}
1 & 0 \\
c & 1
\end{bmatrix} (\tau)
=
\Big( g \,\Big|_{2k} W_{\tilde{c}} \Big)
\Bigg| 
\begin{bmatrix}
1 & \hat{c} \\
0 & \tilde{c}
\end{bmatrix} (\tau)
=
\tilde{c}^{-k}\,
\big( g \,\big|_{2k} W_{\tilde{c}} \big)\!\left(\frac{\tau + \hat{c}}{\tilde{c}}\right).
\]

Taking $c = l$, we obtain
\[
\phi_s^* f \,\Big|\,
\begin{bmatrix}
1 & 0 \\
l & 1
\end{bmatrix} (\tau)
=
(l')^{-k}\,
\Big(\phi_s^* f \,\Big|_{2k} W_{l'}\Big)\!\left(\frac{\tau}{l'}\right).
\]

Using \cite[ Proposition 4.2]{MR2379329}, we also have
\[
\begin{bmatrix}
I & 0 \\
ls^{-1} & I 
\end{bmatrix}
=
\begin{bmatrix}
s & -\hat{l}I \\
lI & (1-l\hat{l})s^{-1}
\end{bmatrix}
\begin{bmatrix}
s^{-1} & \hat{l}I \\
0 & s
\end{bmatrix},
\]
where 
\[
\begin{bmatrix}
s & -\hat{l}I \\
lI & (1-l\hat{l})s^{-1}
\end{bmatrix} \in \Gamma^2_0(l), 
\qquad l\hat{l} \equiv 1 \pmod{l'}.
\]

Hence,
\[
\begin{aligned}
f \,\Bigg|\, 
\begin{bmatrix}
I & 0 \\
ls^{-1} & I 
\end{bmatrix} (s\tau)
&=
f \,\Bigg|\, 
\begin{bmatrix}
s & -\hat{l}I \\
lI & (1-l\hat{l})s^{-1}
\end{bmatrix}
\begin{bmatrix}
s^{-1} & \hat{l}I \\
0 & s
\end{bmatrix} (s\tau) \\[6pt]
&= \chi\!\big((1-l\hat{l})s^{-1}\big)\,
f \,\Bigg|\, 
\begin{bmatrix}
s^{-1} & \hat{l}I \\
0 & s
\end{bmatrix}(s\tau) \\[6pt]
&= \chi\!\big(\det((1-l\hat{l})s^{-1})\big)\,
\det(s^{-1})^{k}\,
f\big(s^{-1}s\tau + \hat{l}I_2\big)(s^{-1}).
\end{aligned}
\]

Now, replacing $\tau$ by $l'\tau$ and taking $A = s^{-1}$ in Proposition 4.5 of \cite{MR2379329}, we obtain
\begin{equation}
\begin{split}
\big(\phi_s^* f \,\big|_{2k} W_{l'}\big)(\tau) 
&= (l')^k\,\chi\!\big(\det((1-l\hat{l})s^{-1})\big)\,
\det(s^{-1})^{k}\,\phi_{l's^{-1}}^* f(\tau) \\[6pt]
&= \chi\!\big(\det((1-l\hat{l})s^{-1})\big)\,\phi_{l's^{-1}}^* f(\tau) \\[6pt]
&= \chi\!\big(\det((1-l\hat{l})s^{-1})\big)\,\phi_s^* f(\tau),
\end{split}
\end{equation}
since $\det(s^{-1}) = (l')^{-1}$ and $l's^{-1}$ is properly equivalent to $s$.
\end{proof}

\begin{remark}
\label{remarkconj}
Using Proposition 4.5, observe that if $l'=\det(s)$ is not a quadratic residue modulo $l$, then, for $\chi$ satisfying the hypotheses of Conjecture 1.2, $\chi(l')^{-1}$ is not an eigenvalue of $W_{l'}$, and hence $\phi_s^*f\equiv 0$.
\end{remark}

\begin{corollary}
$W_{ll'}\phi_s^* f(\tau)=\chi(l')W_{l}\phi_s^* f(\tau)$
\end{corollary}

\begin{proof}
        \begin{equation}
\begin{split}
W_{ll'}\phi_s^* f(\tau)
&= \phi_s^* f\big|_{W_{ll'}}(\tau) \\
&= \phi_s^* f\Bigg|_{\begin{pmatrix}
0 & -\tfrac{1}{\sqrt{ll'}} \\[4pt]
\sqrt{ll'} & 0
\end{pmatrix}}(\tau) \\
&= f\Bigg|_{\begin{pmatrix}
0 & -\tfrac{s}{\sqrt{ll'}} \\[4pt]
\sqrt{ll'}s^{-1} & 0
\end{pmatrix}}(s\tau) \\
&= f\Bigg|_{E_2
\begin{pmatrix}
\sqrt{ll'}s^{-1} & 0 \\[4pt]
0 & \tfrac{s}{\sqrt{ll'}}
\end{pmatrix}}(s\tau) \\
&= \bigl(\det(\tfrac{s}{\sqrt{ll'}})\bigr)^{-2}
\, f\big|_{E_2}(ll'\tau I_2s^{-1}) \\
&= \left(\tfrac{1}{l}\right)^{-2}
\, f\big|_{E_2}(l's^{-1}l\tau I_2) \\
&= \left({l}\right)^{2}
\, \phi_{l's^{-1}}^*
\!\bigl(f\big|_{E_2}\bigr)(l\tau) \\
&\overset{*}{=}
\left({l}\right)^{2}
\, \phi_s^*
\!\bigl(f\big|_{E_2}\bigr)(l\tau) \\
&=
\left({l}\right)^{2}
\, \phi_{ls}^*
\!\bigl(f\big|_{E_2}\bigr)(\tau).
\end{split}
\end{equation}
The statement follows from Proposition 4.5 and \cite[Proposition 1.4]{MR508986}
\end{proof}

\section{\textbf{Proof of the Main Theorem}}

Using the results established in the previous section, we now present an algorithm to compute the Fourier coefficients of the elliptic cusp forms that appear in the image of the restriction maps. 
This enables us to determine an upper bound for the dimension of the space of Siegel modular forms of level $l$ with non-trivial nebentypus, $l, l'$ are distinct primes. This algorithm corresponds to the statement of the main theorem.

 For $f \in S_2^2(\Gamma_0^{(2)}(l), \chi)$, proceed as follows:

\begin{enumerate}[label=\textbf{Step \arabic*:}]

    \item Begin with the choice \(s=I_2\). Apply the restriction map \(\phi_s^*\) to \(f\) and compute the resulting elliptic modular form
    \(\phi_s^*(f)\). By construction,
    \[
    \phi_s^*(f)\in S_1^4(\Gamma_0(l),\chi),
    \]
    and its Fourier coefficients are determined explicitly by those of \(f\).
    Compute the action of the Fricke involution \(W_l\) on
    \(\phi_s^*(f)\). Comparing the Fourier coefficients of
    \(\phi_s^*(f)\) and \(W_l\phi_s^*(f)\) with respect to the chosen basis
    of the relevant elliptic eigenspace yields a system of linear equations
    in the Fourier coefficients of \(f\).

    \vspace{1.0em}

    \item Choose \(s\in V_2(\mathbb{Z})\) such that
    \[
    l'=\det(s)
    \]
    is prime, \(\gcd(l,l')=1\), and \(l'\) is a quadratic residue modulo
    \(l\). Apply the restriction map \(\phi_s^*\) to \(f\) and compute
    \(\phi_s^*(f)\). Then
    \[
    \phi_s^*(f)\in S_1^4(\Gamma_0(ll'),\chi'),
    \]
    where \(l'=\det(s)\), and the Fourier coefficients of
    \(\phi_s^*(f)\) are explicitly determined by those of \(f\).
    Using \textsf{SageMath}, determine the dimension of
    \(S_1^4(\Gamma_0(ll'),\chi')\), express \(\phi_s^*(f)\) as a linear
    combination of a chosen basis of this space, and compare the resulting
    Fourier coefficients to obtain further linear constraints.

    \vspace{1.0em}

    \item Compute the action of the Fricke involution \(W_{l'}\) on
    \(\phi_s^*(f)\), using the results established in the preceding section.
    In particular, verify that
    \[
    \phi_s^*(f)\in
    \ker\left(W_{l'}-\chi(l'^{-1})I\right).
    \]
    Combine this eigenvalue condition with the previously obtained
    vanishing conditions to construct the corresponding subspaces.

    \vspace{1.0em}

    \item Repeat the above procedure for a suitable collection of choices
    of \(s\). Each choice yields an additional system of linear equations,
    with the unknowns given by the Fourier coefficients of \(f\).

    \vspace{1.0em}

    \item Combine all the resulting systems of linear equations and solve
    the resulting linear system. This determines the linear dependencies
    among the Fourier coefficients of \(f\) and identifies a set of
    independent variables, thereby determining the corresponding
    subspace of forms satisfying all the imposed conditions.

\end{enumerate}

The minimal set of independent variables obtained through the above procedure provides an upper bound for the dimension of the Siegel eigenspace $S_2^2(\Gamma_0^{(2)}(l), \chi)$.

\section{Application of the Algorithm to compute the Upper bound.}

In this section, we give a brief demosntration of our algorithm using an example with a prime level and non-trivial nebentypus. Using the results obtained in the previous sections, we compute the corresponding Fourier coefficients and the dimensions of the associated spaces.

To the best of our knowledge, the action of Atkin-Lehner operators on the space of Siegel modular forms with character is not implemented in any existing mathematical software. 
Therefore, to carry out the corresponding computations, we employ the space of modular symbols, since the action of these operators is available in \texttt{Sage}.

Consider the vector space  $S_2^2(\Gamma_0^{(2)}(11),\chi)$ where $\chi$   is the character induced by the Dirichlet character $\tilde{\chi}:(\mathbb{Z}/{11}\mathbb{Z})^{\times} \rightarrow \mathbb{C}^\times$ where $\tilde{\chi}(2)= -\zeta_{10}^3$ completely determines its values. 

The set of determining coefficients is given by

\[
\begin{aligned}
&a_0(2^02),\quad
a_0(2^04),\quad
a_0(2^12),\quad
a_0(2^14),\quad
a_0(2^16),\\
&a_0(4^14),\quad
a_0(4^24),\\[4pt]
&a_2\left(\frac{2^02}{11}\right),\quad
a_2\left(\frac{2^04}{11}\right),\quad
a_2\left(\frac{2^12}{11}\right),\quad
a_2\left(\frac{2^14}{11}\right),\quad
a_2\left(\frac{2^16}{11}\right),\\
&a_2\left(\frac{4^14}{11}\right),\quad
a_2\left(\frac{4^24}{11}\right).
\end{aligned}
\]
In the LFMFDB, we are looking at $11.4.c.a$. Choosing the matrices $s=\begin{pmatrix}
    1 & 0 \\
    0 & 1
\end{pmatrix}$ we have, 

\[
\begin{aligned}
(\phi_s^*f)(\tau)=&
\big(a_0(2^02) + 2a_0(2^12)\big)q^{2} \\[6pt]
&+ \big(4a_0(2^02) + 2a_0(2^04) + 4a_0(2^14)\big)q^{3} \\[6pt]
&+ \big(4a_0(2^04) + 2a_0(2^06) + 4a_0(2^12) + 2a_0(2^14) + 4a_0(2^16) + a_0(4^04) + 2a_0(4^14) + 2a_0(4^24)\big)q^{4} \\[6pt]
&+ \big(4a_0(2^04) + 4a_0(2^06) + 2a_0(2^08) + 4a_0(2^14) + 4a_0(2^18) + 2a_0(4^06) + 4a_0(4^14) + 4a_0(4^16) + 4a_0(4^26)\big)q^{5} \\[6pt]
&+ \big(4a_0(2^02) + 4a_0(2^08) + 4a_0(2^14) + 6a_0(2^16) + 4a_0(4^04) + 4a_0(4^16) + 2a_0(4^26) + 4a_0(4^28) + 2a_0(6^26) + 2a_0(6^36)\big)q^{6} \\[6pt]
\end{aligned}
\]
And 
\[
\begin{aligned}
W_{11}(\phi_s^*f)(\tau)=&11^{2}
\left(a_2\!\left(\frac{2^02}{11}\right)
+2a_2\!\left(\frac{2^12}{11}\right)\right)q^2\\[6pt]
&+11^{2}\left(4a_2\!\left(\frac{2^02}{11}\right)
+2a_2\!\left(\frac{2^04}{11}\right)
+4a_2\!\left(\frac{2^14}{11}\right)\right)q^3\\[6pt]
&+11^{2}\left(4a_2\!\left(\frac{2^04}{11}\right)
+2a_2\!\left(\frac{2^06}{11}\right)
+4a_2\!\left(\frac{2^12}{11}\right)
+2a_2\!\left(\frac{2^14}{11}\right)
+4a_2\!\left(\frac{2^16}{11}\right)\right.\\
&\qquad\qquad\left.
+a_2\!\left(\frac{4^04}{11}\right)
+2a_2\!\left(\frac{4^14}{11}\right)
+2a_2\!\left(\frac{4^24}{11}\right)\right)q^4\\[6pt]
&+11^{2}\left(4a_2\!\left(\frac{2^04}{11}\right)
+4a_2\!\left(\frac{2^06}{11}\right)
+4a_2\!\left(\frac{2^14}{11}\right)\right.\\
&\qquad\qquad\left.
+4a_2\!\left(\frac{2^18}{11}\right)
+4a_2\!\left(\frac{4^14}{11}\right)
+4a_2\!\left(\frac{4^16}{11}\right)
+4a_2\!\left(\frac{4^26}{11}\right)\right)q^5\\[6pt]
&+11^{2}\left(4a_2\!\left(\frac{2^02}{11}\right)
+4a_2\!\left(\frac{2^14}{11}\right)
+6a_2\!\left(\frac{2^16}{11}\right)\right.\\
&\qquad\qquad\left.
+4a_2\!\left(\frac{4^04}{11}\right)
+4a_2\!\left(\frac{4^16}{11}\right)
+2a_2\!\left(\frac{4^26}{11}\right)
+2a_2\!\left(\frac{6^36}{11}\right)\right)q^6
+\cdots
\end{aligned}
\]

Since the dimension of the cusp forms with chosen character is 2 and only one basis element starts with $q^2$, this suggests us to put $\phi_s^*f=cf_1$ and  where
\[
f_1=q^2
-\zeta_{10}^3q^3+\left(3\zeta_{10}^3-\zeta_{10}^2+\zeta_{10}-3\right)q^4
+\left(-3\zeta_{10}^3+2\zeta_{10}^2-3\zeta_{10}\right)q^5
+\left(\zeta_{10}^3+2\zeta_{10}^2+\zeta_{10}\right)q^6
+O(q^{7}).
\]
$$W_{11}\phi_s^*f=\overline{cf_1}$$

Comparing coefficients we have 
\[
\begin{aligned}
&\zeta_{10}^{3}\big(2a_0(2^02)+a_0(2^04)+2a_0(2^14)\big)\\
&\qquad +(3\zeta_{10}^{3}-\zeta_{10}^{2}+\zeta_{10}-3)
\big(a_0(2^02)+2a_0(2^12)\big)=0,
\\[6pt]
&\zeta_{10}^{3}\big(2a_0(2^02)+3a_0(2^04)+a_0(2^06)
+2a_0(2^12)+2a_0(2^14)+2a_0(2^16)\big)\\
&\qquad +(-3\zeta_{10}^{3}+2\zeta_{10}^{2}-3\zeta_{10})
\big(a_0(2^02)+2a_0(2^12)\big)=0.
\end{aligned}
\]

Comparing coefficients for $W_{11}\phi_s^*f$ 
\[
\begin{aligned}
&4a_2\!\left(\frac{2^02}{11}\right)
+2a_2\!\left(\frac{2^04}{11}\right)
+4a_2\!\left(\frac{2^14}{11}\right)\\
&\qquad
+\zeta_{10}^{7}\left(
a_2\!\left(\frac{2^02}{11}\right)
+2a_2\!\left(\frac{2^12}{11}\right)
\right)=0,
\\[8pt]
&4a_2\!\left(\frac{2^04}{11}\right)
+2a_2\!\left(\frac{2^06}{11}\right)
+4a_2\!\left(\frac{2^12}{11}\right)
+2a_2\!\left(\frac{2^14}{11}\right)
+4a_2\!\left(\frac{2^16}{11}\right)\\
&\qquad
+a_2\!\left(\frac{4^04}{11}\right)
+2a_2\!\left(\frac{4^14}{11}\right)
+2a_2\!\left(\frac{4^24}{11}\right)\\
&\qquad
-\left(3\zeta_{10}^{7}-\zeta_{10}^{8}
+\zeta_{10}^{9}-3\right)
\left(
a_2\!\left(\frac{2^02}{11}\right)
+2a_2\!\left(\frac{2^12}{11}\right)
\right)=0.
\end{aligned}
\]

For $s=\begin{pmatrix}2&1\\1&2\end{pmatrix}$ we have $\det s=3$, so $\phi_s^*f\in S_6\big(\Gamma_0(33),\chi\big)$, where $\chi$ is the character modulo $11$ with $\chi(2)=-\zeta$ where $\zeta=\zeta_{10}$, induced to modulus $33$ (order $5$, conductor $11$), and
\[
\dim S_6\big(\Gamma_0(33),\chi\big)=18 .
\]
The smallest exponent in $\phi_s^*f$ is
\[
\min_{T>0}\operatorname{tr}(Ts)=\operatorname{tr}\!\left(\tfrac12\begin{pmatrix}2&-1\\-1&2\end{pmatrix}\begin{pmatrix}2&1\\1&2\end{pmatrix}\right)=3 ,
\]
and the same holds at every cusp (Corollary 4.5), so $\phi_s^*f$ vanishes to order at least $3$ at each of the cusps $W_1\infty,\,W_3\infty,\,W_{11}\infty,\,W_{33}\infty$.
Further, $W_3\,\phi_s^*f=-\zeta\,\phi_s^*f$. Imposing these conditions one at a time gives
\begin{align*}
S_6\big(\Gamma_0(33),\chi\big) &: 18\\
\{W_3f=-\zeta f\} &: 9\\
\operatorname{ord}_{\infty}(f)\ge 3 &: 7\\
\operatorname{ord}_{\infty}(f|W_3)\ge 3 &: 7\\
\operatorname{ord}_{\infty}(f|W_{11})\ge 3 &: 5\\
\operatorname{ord}_{\infty}(f|W_{33})\ge 3 &: 5 .
\end{align*}
Hence $\phi_s^*f$ lies in a $5$-dimensional subspace $V_{33}$ with echelon basis
\begin{align*}
f_1&=q^{3}+(4-4\zeta+13\zeta^{2}-9\zeta^{3})q^{8}+O(q^{9}),\\
f_2&=q^{4}+(1-2\zeta+8\zeta^{2}-5\zeta^{3})q^{8}+O(q^{9}),\\
f_3&=q^{5}+(-\zeta+4\zeta^{2}-4\zeta^{3})q^{8}+O(q^{9}),\\
f_4&=q^{6}+(1-\zeta+3\zeta^{2}-\zeta^{3})q^{8}+O(q^{9}),\\
f_5&=q^{7}+(\zeta^{2}-\zeta^{3})q^{8}+O(q^{9}).
\end{align*}
Writing $\phi_s^*f=\sum_{i=1}^{5}c_if_i$ and $W_{11}\phi_s^*f=\sum_{i=1}^{5}c_i\,W_{11}f_i$, comparing the coefficients of $q^{1},\dots,q^{8}$ in both expansions, and eliminating $c_1,\dots,c_5$ gives the equations $E_{2.1},\dots,E_{2.7}$ (rank $7$).

For $s=\begin{pmatrix}2&1\\1&2\end{pmatrix}$ the equations in $a_0$ and $a_2$ are as follows:
\begin{align*}
E_{2.1}:&\quad (-\tfrac{6}{11}+\tfrac{27}{11}\zeta-\tfrac{18}{11}\zeta^{2}+\tfrac{36}{11}\zeta^{3})a_0(2^{0}2)\\
&\qquad + (-\tfrac{16}{11}+\tfrac{35}{11}\zeta-\tfrac{29}{11}\zeta^{2}+\tfrac{49}{11}\zeta^{3})a_0(2^{1}2)\\
&\qquad + (-\tfrac{6}{11}+\tfrac{6}{11}\zeta-\tfrac{9}{11}\zeta^{2}+\tfrac{12}{11}\zeta^{3})a_0(2^{1}4)\\
&\qquad + (6+3\zeta+6\zeta^{3})a_2\!\left(\tfrac{2^{0}2}{11}\right)\\
&\qquad + (10+2\zeta+\zeta^{2}+5\zeta^{3})a_2\!\left(\tfrac{2^{1}2}{11}\right)\\
&\qquad + 3a_2\!\left(\tfrac{2^{1}4}{11}\right)=0\\[4pt]
E_{2.2}:&\quad (\tfrac{87}{5}-\tfrac{39}{5}\zeta-\tfrac{9}{5}\zeta^{2}-\tfrac{48}{5}\zeta^{3})a_0(2^{0}2)\\
&\qquad + 6a_0(2^{0}4)\\
&\qquad + (\tfrac{128}{5}-\tfrac{56}{5}\zeta-\tfrac{6}{5}\zeta^{2}-\tfrac{87}{5}\zeta^{3})a_0(2^{1}2)\\
&\qquad + (9-3\zeta-6\zeta^{3})a_0(2^{1}4) + a_0(4^{2}4)\\
&\qquad + (\tfrac{33}{5}-\tfrac{231}{5}\zeta-\tfrac{66}{5}\zeta^{2}-\tfrac{132}{5}\zeta^{3})a_2\!\left(\tfrac{2^{0}2}{11}\right)\\
&\qquad + (\tfrac{77}{5}-\tfrac{209}{5}\zeta-\tfrac{99}{5}\zeta^{2}-\tfrac{88}{5}\zeta^{3})a_2\!\left(\tfrac{2^{1}2}{11}\right)=0\\[4pt]
E_{2.3}:&\quad (\tfrac{177}{55}-\tfrac{54}{55}\zeta+\tfrac{201}{55}\zeta^{2}-\tfrac{303}{55}\zeta^{3})a_0(2^{0}2)\\
&\qquad + (\tfrac{403}{55}-\tfrac{91}{55}\zeta+\tfrac{349}{55}\zeta^{2}-\tfrac{37}{5}\zeta^{3})a_0(2^{1}2)\\
&\qquad + (\tfrac{24}{11}-\tfrac{3}{11}\zeta+\tfrac{27}{11}\zeta^{2}-\tfrac{24}{11}\zeta^{3})a_0(2^{1}4)\\
&\qquad + (\tfrac{63}{5}-\tfrac{36}{5}\zeta+\tfrac{9}{5}\zeta^{2}-\tfrac{102}{5}\zeta^{3})a_2\!\left(\tfrac{2^{0}2}{11}\right)\\
&\qquad + 6a_2\!\left(\tfrac{2^{0}4}{11}\right)\\
&\qquad + (\tfrac{42}{5}-\tfrac{29}{5}\zeta+\tfrac{1}{5}\zeta^{2}-\tfrac{93}{5}\zeta^{3})a_2\!\left(\tfrac{2^{1}2}{11}\right)\\
&\qquad + a_2\!\left(\tfrac{4^{2}4}{11}\right)=0\\[4pt]
E_{2.4}:&\quad (-\tfrac{36}{5}+\tfrac{132}{5}\zeta-\tfrac{33}{5}\zeta^{2}+\tfrac{144}{5}\zeta^{3})a_0(2^{0}2)\\
&\qquad + (-\tfrac{119}{5}+\tfrac{173}{5}\zeta-\tfrac{47}{5}\zeta^{2}+\tfrac{236}{5}\zeta^{3})a_0(2^{1}2)\\
&\qquad + (9\zeta-3\zeta^{2}+15\zeta^{3})a_0(2^{1}4) + 3a_0(2^{1}6) + 3a_0(4^{1}4)\\
&\qquad + (\tfrac{66}{5}+\tfrac{528}{5}\zeta+\tfrac{33}{5}\zeta^{2}+\tfrac{396}{5}\zeta^{3})a_2\!\left(\tfrac{2^{0}2}{11}\right)\\
&\qquad + (\tfrac{209}{5}+\tfrac{242}{5}\zeta+\tfrac{187}{5}\zeta^{2}+\tfrac{154}{5}\zeta^{3})a_2\!\left(\tfrac{2^{1}2}{11}\right)=0\\[4pt]
E_{2.5}:&\quad (-\tfrac{186}{55}-\tfrac{483}{55}\zeta-\tfrac{63}{55}\zeta^{2}-\tfrac{171}{55}\zeta^{3})a_0(2^{0}2)\\
&\qquad + (-\tfrac{464}{55}-\tfrac{32}{5}\zeta-\tfrac{197}{55}\zeta^{2}-\tfrac{124}{55}\zeta^{3})a_0(2^{1}2)\\
&\qquad + (-\tfrac{30}{11}-\tfrac{18}{11}\zeta-\tfrac{15}{11}\zeta^{2}-\tfrac{9}{11}\zeta^{3})a_0(2^{1}4)\\
&\qquad + (-\tfrac{114}{5}+\tfrac{3}{5}\zeta-\tfrac{27}{5}\zeta^{2}+\tfrac{81}{5}\zeta^{3})a_2\!\left(\tfrac{2^{0}2}{11}\right)\\
&\qquad + (-\tfrac{161}{5}+\tfrac{22}{5}\zeta-\tfrac{43}{5}\zeta^{2}+\tfrac{89}{5}\zeta^{3})a_2\!\left(\tfrac{2^{1}2}{11}\right)\\
&\qquad + 6a_2\!\left(\tfrac{2^{1}4}{11}\right) + 3a_2\!\left(\tfrac{2^{1}6}{11}\right)\\
&\qquad + 3a_2\!\left(\tfrac{4^{1}4}{11}\right)=0\\[4pt]
E_{2.6}:&\quad (33-21\zeta+48\zeta^{2}-21\zeta^{3})a_0(2^{0}2) + 6a_0(2^{0}6)\\
&\qquad + (35-19\zeta+57\zeta^{2}-15\zeta^{3})a_0(2^{1}2) + (12-3\zeta+15\zeta^{2})a_0(2^{1}4) + 3a_0(4^{0}4)\\
&\qquad + 3a_0(4^{2}6) + (66-66\zeta+132\zeta^{2}-132\zeta^{3})a_2\!\left(\tfrac{2^{0}2}{11}\right)\\
&\qquad + (66-55\zeta+121\zeta^{2}-143\zeta^{3})a_2\!\left(\tfrac{2^{1}2}{11}\right)=0\\[4pt]
E_{2.7}:&\quad (-9\zeta+3\zeta^{2}-9\zeta^{3})a_0(2^{0}2)\\
&\qquad + (\tfrac{64}{11}-\tfrac{168}{11}\zeta+\tfrac{62}{11}\zeta^{2}-\tfrac{151}{11}\zeta^{3})a_0(2^{1}2)\\
&\qquad + (\tfrac{39}{11}-\tfrac{57}{11}\zeta+\tfrac{12}{11}\zeta^{2}-\tfrac{42}{11}\zeta^{3})a_0(2^{1}4)\\
&\qquad + (-18-12\zeta-3\zeta^{2}-12\zeta^{3})a_2\!\left(\tfrac{2^{0}2}{11}\right)\\
&\qquad + 6a_2\!\left(\tfrac{2^{0}6}{11}\right)\\
&\qquad + (-27-13\zeta-3\zeta^{2}-8\zeta^{3})a_2\!\left(\tfrac{2^{1}2}{11}\right)\\
&\qquad + 3a_2\!\left(\tfrac{4^{0}4}{11}\right) + 3a_2\!\left(\tfrac{4^{2}6}{11}\right)=0
\end{align*}

For $s=\begin{pmatrix}2&1\\1&3\end{pmatrix}$ the equations in $a_0$ and $a_2$ are as follows:
\begin{align*}
E_{3.1}:&\quad (\tfrac{183}{5}-\tfrac{11}{5}\zeta-\tfrac{36}{5}\zeta^{2}-\tfrac{47}{5}\zeta^{3})a_0(2^{0}2)\\
&\qquad + (\tfrac{36}{5}-\tfrac{2}{5}\zeta-\tfrac{2}{5}\zeta^{2}-\tfrac{4}{5}\zeta^{3})a_0(2^{0}4)\\
&\qquad + (\tfrac{231}{5}-\tfrac{12}{5}\zeta-\tfrac{52}{5}\zeta^{2}-\tfrac{59}{5}\zeta^{3})a_0(2^{1}2)\\
&\qquad + (\tfrac{69}{5}-\tfrac{3}{5}\zeta-\tfrac{8}{5}\zeta^{2}-\tfrac{11}{5}\zeta^{3})a_0(2^{1}4)\\
&\qquad + a_0(2^{1}6) + a_0(4^{2}4)\\
&\qquad + (\tfrac{332}{5}-\tfrac{364}{5}\zeta-\tfrac{169}{5}\zeta^{2}-\tfrac{158}{5}\zeta^{3})a_2\!\left(\tfrac{2^{0}2}{11}\right)\\
&\qquad + (\tfrac{14}{5}-\tfrac{18}{5}\zeta-\tfrac{28}{5}\zeta^{2}-\tfrac{6}{5}\zeta^{3})a_2\!\left(\tfrac{2^{0}4}{11}\right)\\
&\qquad + (\tfrac{484}{5}-\tfrac{528}{5}\zeta-\tfrac{198}{5}\zeta^{2}-\tfrac{231}{5}\zeta^{3})a_2\!\left(\tfrac{2^{1}2}{11}\right)\\
&\qquad + (\tfrac{56}{5}-\tfrac{72}{5}\zeta-\tfrac{57}{5}\zeta^{2}-\tfrac{24}{5}\zeta^{3})a_2\!\left(\tfrac{2^{1}4}{11}\right)=0\\[4pt]
E_{3.2}:&\quad (-\tfrac{32}{55}+\tfrac{364}{55}\zeta-\tfrac{206}{55}\zeta^{2}+\tfrac{533}{55}\zeta^{3})a_0(2^{0}2)\\
&\qquad + (-\tfrac{4}{55}+\tfrac{18}{55}\zeta-\tfrac{12}{55}\zeta^{2}+\tfrac{46}{55}\zeta^{3})a_0(2^{0}4)\\
&\qquad + (-\tfrac{4}{5}+\tfrac{48}{5}\zeta-\tfrac{27}{5}\zeta^{2}+\tfrac{66}{5}\zeta^{3})a_0(2^{1}2)\\
&\qquad + (-\tfrac{16}{55}+\tfrac{72}{55}\zeta-\tfrac{48}{55}\zeta^{2}+\tfrac{129}{55}\zeta^{3})a_0(2^{1}4)\\
&\qquad + (\tfrac{172}{5}+\tfrac{11}{5}\zeta+\tfrac{36}{5}\zeta^{2}+\tfrac{47}{5}\zeta^{3})a_2\!\left(\tfrac{2^{0}2}{11}\right)\\
&\qquad + (\tfrac{34}{5}+\tfrac{2}{5}\zeta+\tfrac{2}{5}\zeta^{2}+\tfrac{4}{5}\zeta^{3})a_2\!\left(\tfrac{2^{0}4}{11}\right)\\
&\qquad + (\tfrac{219}{5}+\tfrac{12}{5}\zeta+\tfrac{47}{5}\zeta^{2}+\tfrac{64}{5}\zeta^{3})a_2\!\left(\tfrac{2^{1}2}{11}\right)\\
&\qquad + (\tfrac{66}{5}+\tfrac{3}{5}\zeta+\tfrac{8}{5}\zeta^{2}+\tfrac{11}{5}\zeta^{3})a_2\!\left(\tfrac{2^{1}4}{11}\right)\\
&\qquad + a_2\!\left(\tfrac{2^{1}6}{11}\right) + a_2\!\left(\tfrac{4^{2}4}{11}\right)=0\\[4pt]
E_{3.3}:&\quad (-\tfrac{2}{5}+\tfrac{4}{5}\zeta+\tfrac{49}{5}\zeta^{2}+\tfrac{63}{5}\zeta^{3})a_0(2^{0}2)\\
&\qquad + (\tfrac{16}{5}+\tfrac{8}{5}\zeta-\tfrac{2}{5}\zeta^{2}+\tfrac{6}{5}\zeta^{3})a_0(2^{0}4)\\
&\qquad + 2a_0(2^{0}6)\\
&\qquad + (-\tfrac{21}{5}-\tfrac{13}{5}\zeta+\tfrac{87}{5}\zeta^{2}+\tfrac{74}{5}\zeta^{3})a_0(2^{1}2)\\
&\qquad + (\tfrac{18}{5}+\tfrac{4}{5}\zeta+\tfrac{9}{5}\zeta^{2}+\tfrac{13}{5}\zeta^{3})a_0(2^{1}4)\\
&\qquad + 2a_0(4^{1}4)\\
&\qquad + (\tfrac{2}{5}+\tfrac{76}{5}\zeta+\tfrac{381}{5}\zeta^{2}+\tfrac{7}{5}\zeta^{3})a_2\!\left(\tfrac{2^{0}2}{11}\right)\\
&\qquad + (\tfrac{14}{5}-\tfrac{18}{5}\zeta+\tfrac{82}{5}\zeta^{2}-\tfrac{6}{5}\zeta^{3})a_2\!\left(\tfrac{2^{0}4}{11}\right)\\
&\qquad + (\tfrac{66}{5}+\tfrac{143}{5}\zeta+\tfrac{473}{5}\zeta^{2}-\tfrac{44}{5}\zeta^{3})a_2\!\left(\tfrac{2^{1}2}{11}\right)\\
&\qquad + (\tfrac{12}{5}+\tfrac{16}{5}\zeta+\tfrac{141}{5}\zeta^{2}-\tfrac{13}{5}\zeta^{3})a_2\!\left(\tfrac{2^{1}4}{11}\right)=0\\[4pt]
E_{3.4}:&\quad (\tfrac{78}{55}-\tfrac{76}{55}\zeta+\tfrac{69}{55}\zeta^{2}-\tfrac{457}{55}\zeta^{3})a_0(2^{0}2)\\
&\qquad + (-\tfrac{4}{55}+\tfrac{18}{55}\zeta-\tfrac{12}{55}\zeta^{2}-\tfrac{64}{55}\zeta^{3})a_0(2^{0}4)\\
&\qquad + (\tfrac{19}{5}-\tfrac{13}{5}\zeta+\tfrac{17}{5}\zeta^{2}-\tfrac{56}{5}\zeta^{3})a_0(2^{1}2)\\
&\qquad + (\tfrac{28}{55}-\tfrac{16}{55}\zeta+\tfrac{29}{55}\zeta^{2}-\tfrac{157}{55}\zeta^{3})a_0(2^{1}4)\\
&\qquad + (\tfrac{2}{5}-\tfrac{4}{5}\zeta-\tfrac{59}{5}\zeta^{2}-\tfrac{53}{5}\zeta^{3})a_2\!\left(\tfrac{2^{0}2}{11}\right)\\
&\qquad + (\tfrac{24}{5}-\tfrac{8}{5}\zeta+\tfrac{2}{5}\zeta^{2}-\tfrac{6}{5}\zeta^{3})a_2\!\left(\tfrac{2^{0}4}{11}\right)\\
&\qquad + 2a_2\!\left(\tfrac{2^{0}6}{11}\right)\\
&\qquad + (-\tfrac{34}{5}+\tfrac{13}{5}\zeta-\tfrac{87}{5}\zeta^{2}-\tfrac{74}{5}\zeta^{3})a_2\!\left(\tfrac{2^{1}2}{11}\right)\\
&\qquad + (\tfrac{22}{5}-\tfrac{4}{5}\zeta-\tfrac{9}{5}\zeta^{2}-\tfrac{13}{5}\zeta^{3})a_2\!\left(\tfrac{2^{1}4}{11}\right)\\
&\qquad + 2a_2\!\left(\tfrac{4^{1}4}{11}\right)=0
\end{align*}

For $s=\begin{pmatrix}1&0\\0&2\end{pmatrix}$ the equations in $a_0$ and $a_2$ are as follows:
\begin{align*}
E_{4.1}:&\quad (\tfrac{21}{5}-\tfrac{2}{5}\zeta+\tfrac{3}{5}\zeta^{2}+\tfrac{1}{5}\zeta^{3})a_0(2^{0}2)\\
&\qquad + a_0(2^{0}4)\\
&\qquad + (\tfrac{22}{5}-\tfrac{4}{5}\zeta+\tfrac{6}{5}\zeta^{2}+\tfrac{2}{5}\zeta^{3})a_0(2^{1}2)\\
&\qquad + 2a_0(2^{1}4)\\
&\qquad + (\tfrac{44}{5}-\tfrac{33}{5}\zeta+\tfrac{22}{5}\zeta^{2}-\tfrac{11}{5}\zeta^{3})a_2\!\left(\tfrac{2^{0}2}{11}\right)\\
&\qquad + (\tfrac{88}{5}-\tfrac{66}{5}\zeta+\tfrac{44}{5}\zeta^{2}-\tfrac{22}{5}\zeta^{3})a_2\!\left(\tfrac{2^{1}2}{11}\right)=0\\[4pt]
E_{4.2}:&\quad (\tfrac{1}{5}+\tfrac{3}{5}\zeta-\tfrac{2}{5}\zeta^{2}+\tfrac{1}{5}\zeta^{3})a_0(2^{0}2)\\
&\qquad + (\tfrac{2}{5}+\tfrac{6}{5}\zeta-\tfrac{4}{5}\zeta^{2}+\tfrac{2}{5}\zeta^{3})a_0(2^{1}2)\\
&\qquad + (\tfrac{19}{5}+\tfrac{2}{5}\zeta-\tfrac{3}{5}\zeta^{2}-\tfrac{1}{5}\zeta^{3})a_2\!\left(\tfrac{2^{0}2}{11}\right)\\
&\qquad + a_2\!\left(\tfrac{2^{0}4}{11}\right)\\
&\qquad + (\tfrac{18}{5}+\tfrac{4}{5}\zeta-\tfrac{6}{5}\zeta^{2}-\tfrac{2}{5}\zeta^{3})a_2\!\left(\tfrac{2^{1}2}{11}\right)\\
&\qquad + 2a_2\!\left(\tfrac{2^{1}4}{11}\right)=0\\[4pt]
E_{4.3}:&\quad (\tfrac{28}{5}+\tfrac{4}{5}\zeta-\tfrac{1}{5}\zeta^{2}-\tfrac{7}{5}\zeta^{3})a_0(2^{0}2)\\
&\qquad + 3a_0(2^{0}4) + a_0(2^{0}6)\\
&\qquad + (\tfrac{46}{5}+\tfrac{8}{5}\zeta-\tfrac{2}{5}\zeta^{2}-\tfrac{14}{5}\zeta^{3})a_0(2^{1}2)\\
&\qquad + 2a_0(2^{1}4) + 2a_0(2^{1}6)\\
&\qquad + (-\tfrac{33}{5}+\tfrac{11}{5}\zeta-\tfrac{44}{5}\zeta^{2}-\tfrac{33}{5}\zeta^{3})a_2\!\left(\tfrac{2^{0}2}{11}\right)\\
&\qquad + (-\tfrac{66}{5}+\tfrac{22}{5}\zeta-\tfrac{88}{5}\zeta^{2}-\tfrac{66}{5}\zeta^{3})a_2\!\left(\tfrac{2^{1}2}{11}\right)=0\\[4pt]
E_{4.4}:&\quad (-\tfrac{2}{5}-\tfrac{1}{5}\zeta+\tfrac{4}{5}\zeta^{2}+\tfrac{3}{5}\zeta^{3})a_0(2^{0}2)\\
&\qquad + (-\tfrac{4}{5}-\tfrac{2}{5}\zeta+\tfrac{8}{5}\zeta^{2}+\tfrac{6}{5}\zeta^{3})a_0(2^{1}2)\\
&\qquad + (\tfrac{32}{5}-\tfrac{4}{5}\zeta+\tfrac{11}{5}\zeta^{2}-\tfrac{3}{5}\zeta^{3})a_2\!\left(\tfrac{2^{0}2}{11}\right)\\
&\qquad + 3a_2\!\left(\tfrac{2^{0}4}{11}\right) + a_2\!\left(\tfrac{2^{0}6}{11}\right)\\
&\qquad + (\tfrac{54}{5}-\tfrac{8}{5}\zeta+\tfrac{22}{5}\zeta^{2}-\tfrac{6}{5}\zeta^{3})a_2\!\left(\tfrac{2^{1}2}{11}\right)\\
&\qquad + 2a_2\!\left(\tfrac{2^{1}4}{11}\right) + 2a_2\!\left(\tfrac{2^{1}6}{11}\right)=0
\end{align*}

Combining all the above equations, we have 19 equations in 20 unknowns, and hence $\dim(S_2^2(11,\chi))\leq1$.

Tables~\ref{tab:upper_bounds_characters}, \ref{tab:upper_bounds_wt3} and~\ref{tab:upper_bounds_357} list upper bounds for the dimension at several prime levels, for weights $k=2$, $3$ and $4$ respectively. More computations can be found in ~\cite{debargha2026}.

\begin{table}[ht]
\centering
\small
\renewcommand{\arraystretch}{1.35}
\setlength{\tabcolsep}{5pt}
\setlength{\arrayrulewidth}{0.4pt}
\begin{tabular}{!{\vrule width 1pt}c|c|c|c|c|c|c|c!{\vrule width 1pt}}
\noalign{\hrule height 1pt}
$p$ & $\chi$ & $\left\lfloor \frac{p+1}{3}\right\rfloor$ & $\lVert\mathcal{C}_0\rVert$ & $\beta$ & $\lVert\mathcal{B}_0\rVert$ & $\mathcal{A}$ & $\dim S_2^2(\Gamma_0(p),\chi)\leq$ \\
\noalign{\hrule height 1pt}
$5$ & $\chi(2)=\zeta_4$ & $\frac{3}{2}$ & $2$ & $2$ & $3$ & $2^12$ & $0$ \\
\hline
$7$ & $\chi(3)=\zeta_6-1$ & $\frac{5}{2}$ & $3$ & $2$ & $5$ & $1^01,\ 2^12$ & $0$ \\
\hline
$11$ & $\chi(2)=-\zeta_{10}^{3}$ & $\frac{7}{2}$ & $7$ & $4$ & $9$ & \makecell{$2^12,\ 2^13,$\\ $1^01,\ 1^02$} & $1$ \\
\hline
$13$ & $\chi(2)=\zeta_6-1$ & $\frac{9}{2}$ & $13$ & $6$ & $21$ & \makecell{$2^12,\ 2^13,\ 1^01,$\\ $1^02,\ 3^14$} & $0$ \\
\hline
$17$ & $\chi(3)=-\zeta_4$ & $\frac{11}{2}$ & $21$ & $6.5$ & $42$ & \makecell{$2^12,\ 2^13,\ 1^01,$\\ $1^02,\ 3^14,\ 2^17$} & $0$ \\
\hline
$19$ & $\chi(2)=\zeta_6-1$ & $\frac{13}{2}$ & $32$ & $7.5$ & $46$ & \makecell{$2^12,\ 2^13,\ 1^01,$\\ $3^14,\ 2^19, 2^14$} & $0$ \\
\noalign{\hrule height 1pt}
\end{tabular}
\caption{Upper bounds for $\dim S_2^2(\Gamma_0(p),\chi)$ for $p\leq 19$, non-trivial characters only.}
\label{tab:upper_bounds_characters}
\end{table}

\begin{table}[ht]
\centering
\small
\renewcommand{\arraystretch}{1.6}
\setlength{\tabcolsep}{6pt}
\setlength{\arrayrulewidth}{0.4pt}
\begin{tabular}{!{\vrule width 1pt}c|c|c|c|c|c|c|c!{\vrule width 1pt}}
\noalign{\hrule height 1pt}
$p$ & $\chi$ & $\left\lfloor \frac{2(p+1)}{3}\right\rfloor$ & $\lVert\mathcal{C}_0\rVert$ & $\beta$ & $\lVert\mathcal{B}_0\rVert$ & $\mathcal{A}$ & $\dim S_2^3(\Gamma_0(p),\chi)\leq$ \\
\noalign{\hrule height 1pt}
$5$  & for all char. & 3 & 0 &  &  &  & 0 \\
\hline
$7$  & for all char. & 4 & 0 &  &  &  & 0 \\
\hline
\multirow{3}{*}{$11$}
  &$ \chi(2)=-\zeta_{10}^3$& $5.5$ & $2$ & $7$ & $6$ & $3^15$ & 0 \\
\cline{2-8}
  & $\chi(2)=\zeta_{10}$ & 5.5 & 2 & 7 & 6 & $3^15$ & 0 \\
\cline{2-8}
  & $\chi(2)=-1$ & 5.5 & 2 &  7& 6 & $3^15$ & 0 \\
\noalign{\hrule height 1pt}
\end{tabular}
\caption{Upper bounds for $\dim S_2^3(\Gamma_0(p),\chi)$ for $p\in\{2,3,5,7,11\}$.}
\label{tab:upper_bounds_wt3}
\end{table}

\begin{table}[ht]
\centering
\small
\renewcommand{\arraystretch}{1.6}
\setlength{\tabcolsep}{6pt}
\setlength{\arrayrulewidth}{0.4pt}
\begin{tabular}{!{\vrule width 1pt}c|c|c|c|c|c|c|c!{\vrule width 1pt}}
\noalign{\hrule height 1pt}
$p$ & $\chi$ & $\left\lfloor \frac{2(p+1)}{3}\right\rfloor$ & $\lVert\mathcal{C}_0\rVert$ & $\beta$ & $\lVert\mathcal{B}_0\rVert$ & $\mathcal{A}$ & $\dim S_2^4(\Gamma_0(p),\chi)\leq$ \\
\noalign{\hrule height 1pt}
$3$ & $\chi(2)=-1$ & $2$ & $3$ & $7$ & $5$ & $1^01,2^14,3^14$ & $0$
  \\
\hline
\multirow{2}{*}{$5$}
  & $\chi(2)=-1$ & 4 & 7 & 9 & 27 & $1^01.1^02,2^12,2^14,3^14$ & $1$ \\
\cline{2-8}
  & $\chi(2)=\iota$ & 4 & 7 & 9 & 27 &same as above  & 0 \\
\hline
{$7$}
  & $\chi(3)=\zeta_6-1$ & 5 & 17 & 8 & 57 & $1^01,1^02,3^03,3^14,4^14$ & 2 \\
\noalign{\hrule height 1pt}
\end{tabular}
\caption{Upper bounds for $\dim S_2^4(\Gamma_0(p),\chi)$ for $p\in\{3,5,7\}$.}
\label{tab:upper_bounds_357}
\end{table}

\section{\textbf{Lower Bound}}
The Saito--Kurokawa lift provides a method for lifting classical modular forms to Siegel modular forms. However, the Saito--Kurokawa construction has been established only for even weights $k > 2$, while our interest lies in the case $k = 2$. The Saito--Kurokawa map can be factored as
\[
S_{1}^{k - \frac{1}{2}}(\Gamma_0(4N))^+ \longrightarrow J_{k,1}^{\mathrm{cusp}}(\Gamma_0(N)) \longrightarrow S_{2}^{k}(\Gamma_0(N)),
\]
for odd square-free $N$. The second map in the above composition holds for general $k$.  Poor and Yuen proved in Lemma 6.1 \cite{MR2379329} that the first map will be an isomorphism for $k=2$ for odd square-free $N$. So, the lower bound will be the dimension of the space $S_{1}^{k - \frac{1}{2}}(\Gamma_0(4N))^+$ in the case of a trivial nebentypus.
When working with nontrivial Nebentypus characters, such a factorization is not available. However, injective maps can be constructed that yield lower bounds for the dimensions of the corresponding spaces.  

From Brown and Keaton~\cite{BROWN20131492}, there exists an injective linear map for even weight and even character,
\[
S_{2k-2}^{1}(\Gamma_0(N), \chi^2) \longrightarrow J_{k,1}^{\mathrm{cusp}}(N, \chi),
\]
where $J_{k,1}^{\mathrm{cusp}}(N, \chi)$ denotes the space of Jacobi cusp forms of level~$N$ and $\chi$ is an even Dirichlet character.  

Furthermore, Ibukiyama~\cite[Theorem~3.2]{article} established an injective map
\[
J_{k,1}^{\mathrm{cusp}}(N, \chi) \longrightarrow S_{k}^{2}(\Gamma_0^{(2)}(N), \chi).
\]
Combining these two injections, we obtain a lift
\[
S_{2k-2}^{1}(\Gamma_0(N), \chi^2) \longrightarrow S_{k}^{2}(\Gamma_0^{(2)}(N), \chi).
\]
Consequently, the lower bound for the dimension of the Siegel eigenspace $S_{k}^{2}(\Gamma_0^{(2)}(N), \chi)$ for even weight and even character is given by
\[
\dim S_{k}^{2}(\Gamma_0^{(2)}(N), \chi) \geq \dim S_{2k-2}^{1}(\Gamma_0(N), \chi^2).
\]

\section{Conclusion and Future Work}
\label{sec:conclusion}
Building on the restriction map techniques of Poor and Yuen, we have established explicit upper bounds for the dimensions of spaces of Siegel modular forms with non-trivial even nebentypus of weight $k = 2$, $k=3$, $k=4$ and level $N = ll'$, where $l$ and $l'$ are distinct primes. The computational framework developed here has been demonstrated through representative examples, providing a clear illustration of the underlying method.  

In addition, we have obtained the corresponding lower bounds for these spaces by relating them to classical and Jacobi cusp forms. Together, these results offer a more complete understanding of the dimension theory of Siegel modular forms in low-weight and non-trivial character settings.

\section*{Declaration of AI use}
During the preparation of this work, the author used Claude (Anthropic)~\cite{claude2026}
to assist with the computation of the elliptic subspaces, the restriction expansions
, and the resulting linear systems in Section~6. These computations were carried out
with PARI/GP~\cite{PARI2} and SageMath~\cite{sagemath}. All results were checked
independently by the author, who takes full responsibility for the content of this
paper.
\section{Data Availability}
All generated data can be reproduced using the software SAGE, PARI.

\section{Conflict of Interest}
The authors declare that they have no conflict of interest in regard to the publication of this manuscript with anyone. No financial or non-financial relationships exist that could be perceived as influencing the research presented on the topic.

A few notes:

\bibliographystyle{amsplain}
\bibliography{drondebargha}
\end{document}